\documentclass[]{amsart} 
\usepackage{amssymb}
\usepackage[latin1]{inputenc}
\usepackage{amsgen,amsmath,amsxtra,amssymb,amsfonts,amsthm}
\usepackage{latexsym}
\usepackage{amsmath}
\usepackage{amsthm}
\usepackage{amscd}

\usepackage{amssymb}
\usepackage{amsfonts}

\usepackage{color}
\definecolor{marron}{rgb}{0.25,0,0}

\definecolor{verdenormal}{rgb}{0,0.5,0}

\long\def\salta#1{\relax}

\def\dys{\displaystyle}
\def\vp{\varphi}

\newtheorem{theorem}{Theorem}[section]
\newtheorem{lemma}[theorem]{Lemma}
\newtheorem{proposition}[theorem]{Proposition}

\theoremstyle{definition}
\newtheorem{remark}[theorem]{Remark}

\newtheorem{definition}[theorem]{Definition}

\def\huz{H^1_0 (\Omega)}
\def\dys{\displaystyle}
\def\rife#1{(\ref{#1})}

\def\vp{\varphi}

\newcommand{\na}{\mathbb{N}}

\newcommand{\rn}{\mathbb{R}^{N}}

\catcode`\@=11

\def\@ceqnnum{{\reset@font\rm (\tempequation)}}

\def\@ceqncr{{\ifnum0=`}\fi\@ifstar{\global\@eqpen\@M
 \@cyeqncr}{\global\@eqpen\interdisplaylinepenalty \@cyeqncr}}

\def\@cyeqncr{\@ifnextchar [{\@cxeqncr}{\@cxeqncr[\z@]}}

\def\@cxeqncr[#1]{\ifnum0=`{\fi}\@@ceqncr
 \noalign{\penalty\@eqpen\vskip\jot\vskip #1\relax}}

\def\@@ceqncr{\let\@tempa\relax
 \ifcase\@eqcnt \def\@tempa{& & &}\or \def\@tempa{& &}%
 \else \def\@tempa{&}\fi
 \@tempa \if@eqnsw\@ceqnnum\fi
 \global\@eqnswtrue\global\@eqcnt\z@\cr}

\def\clabel#1#2{\global\let\tempequation #2
 \@bsphack\if@filesw {\let\thepage\relax
 \def\protect{\noexpand\noexpand\noexpand}%
 \edef\@tempa{\write\@auxout{\string
 \newlabel{#1}{{#2}{\thepage}}}}%
 \expandafter}\@tempa
 \if@nobreak \ifvmode\nobreak\fi\fi\fi\@esphack}

\def\ceqnarray{
\global\@eqnswtrue\m@th
\global\@eqcnt\z@\tabskip\@centering\let\\\@ceqncr
$$\halign to\displaywidth\bgroup\@eqnsel\hskip\@centering
 $\displaystyle\tabskip\z@{{}##}$&\global\@eqcnt\@ne
 \hskip 2\arraycolsep \hfil${{}##}$\hfil
 &\global\@eqcnt\tw@ \hskip 2\arraycolsep
$\displaystyle\tabskip\z@{{}##}$\hfil
 \tabskip\@centering&\llap{##}\tabskip\z@\cr}

\def\endceqnarray{\@@ceqncr\egroup$$\global\@ignoretrue}

\catcode`\@=12

\def\vp{\varphi}
\def\be{\begin{equation}}
\def\ee{\end{equation}}
\def\rife#1{(\ref{#1})}

\def\lp{L^{p}(\Omega)}

\def\vare{\varepsilon}

\def\t1p0{T^{1,p}_{0}(\Omega)}

\def\capp{{\rm cap}_{p}}
\def\m2{M^{\frac{N(p-1)}{N-1}}(\Omega)}

\def\div{\text{\rm div}}

\def\into{\int_{\Omega}}
\def\w-1p'{W^{-1,p'}(\Omega)}

\def\dys{\displaystyle}

\def\lp'n{(L^{p'}(\Omega))^{N}}

\author[L. Orsina]{Luigi Orsina}
\email{orsina@mat.uniroma1.it}
\author[F. Petitta]{Francesco Petitta}
\email{francesco.petitta@sbai.uniroma1.it}

\address[L. Orsina]{Dipartimento di Matematica, 
``Sapienza", Universit\`a di Roma, Piazzale A. Moro 2, 00185 Roma, Italy}
\address[F. Petitta]{Dipartimento di Scienze di Base e Applicate
per l' Ingegneria, ``Sapienza", Universit\`a di Roma, Via Scarpa 16, 00161 Roma, Italy.}


\begin{document}
\title{A Lazer-McKenna type problem with measures}

\begin{abstract}
In this paper we are concerned with a general singular Dirichlet boundary value problem whose model is the following
$$
\begin{cases}
 \dys -\Delta u = \frac{\mu}{u^{\gamma}} & \text{in}\ \Omega,\\
 u=0 &\text{on}\ \partial\Omega,\\
 u>0 &\text{on}\ \Omega\,.
 \end{cases}
$$
Here $\mu$ is a nonnegative bounded Radon measure on a bounded open set $\Omega\subset\rn$, and $\gamma>0$. 
\end{abstract}

\maketitle

\tableofcontents

\section{Introduction}
We deal with a general singular Dirichlet boundary value problem whose model is the following
\begin{equation}\label{1}
\begin{cases}
 \dys -\Delta u = \frac{\mu}{u^{\gamma}} & \text{in}\ \Omega,\\
 u=0 &\text{on}\ \partial\Omega,\\
 u>0 &\text{on}\ \Omega,
 \end{cases}
\end{equation}
where $\mu$ is a nonnegative bounded Radon measure on a bounded open set $\Omega\subset\rn$, $N\geq 2$, and $\gamma>0$. 

Problems as in \rife{1} have been extensively studied both for their pure mathematical interest (see \cite{crt} and \cite{lm}) and for relevant connections with some physical phenomena (see \cite{nc}). 

In particular, if $\Omega$ is smooth and $\mu$ is an H\"older continuous function that is strictly positive in $\overline{\Omega}$ then there exists a unique classical solution $u\in C^{2+\alpha}(\Omega)\cap C(\overline{\Omega})$ to problem \rife{1}. However, this solution fails to have finite energy (i.e $u$ belongs to $\huz$) if $\gamma\geq 3$ and it is not $C^{1}(\overline{\Omega})$ if $\gamma>1$ (see \cite{lm,crt}). See also the recent paper \cite{yd} for further details on the role of the exponent $\gamma=3$.

In \cite{bo} the existence of a distributional solution for problem \rife{1} is proved if $\mu$ belongs to some Lebesgue space $L^{m}(\Omega)$ with $m\geq 1$. Moreover, if $\gamma\geq 1$ the solution belongs, at least locally, to $H^{1}(\Omega)$, while for $\gamma<1$ it turns out to be a locally infinite energy solution though it belongs to some larger Sobolev space. 
In the same paper the authors also show that if $\mu$ is singular with respect to some suitable capacity then no solutions are expected to exists at least in the sense of approximating problems. The result is also extended to more general linear operator having a principal part in divergence form. Finally, the case of nonhomogeneous problems is treated in \cite{olpe}.

In this paper we deal with problems as \rife{1} where $\mu$ is a bounded Radon measure which is diffuse with respect to a suitable capacity (depending on the value of $\gamma$). We stress that our results turn out to be sharp with respect to the nonexistence result in \cite{bo} (see Remark \ref{optimal} below). 

The key point in the proof of \cite{bo} was to choose a suitable nondecreasing sequence of approximating solutions in order to get a uniform bound from below for the solution. This allows to give meaning, at least locally, to the right hand side of the equation. When dealing with measures, as pointed out by the same authors, the main difficulty (but not the only one) in order to prove existence of solutions relies on the impossibility to approximate the datum with an increasing sequence of bounded functions. In order to get rid of this problem, in this paper we will construct suitable local barriers from below for the approximating problems, which in turn allow us to provide, as a by-product, a simplified proof of the results known if the datum $\mu$ is not singular with respect to the Lebesgue measure. Our argument, at least in the case of a measure which is not singular with respect to the Lebesgue measure, does not make use of any monotonicity argument. The case of a general diffuse measure, possibly singular with respect to $\mathcal{L}$, is much more delicate and will be treated in Section \ref{sing} via a suitable monotone approximation argument of the datum essentially due to Dal Maso, and Baras and Pierre (\cite{dm,bp}). 
 Let us finally stress that one of the problems to deal with when treating these singular problems consists in providing a suitable definition of solution which has to be regular enough in order to give sense to the term on the right hand side of \rife{1}.
 
 The paper is organized as follows: in the next Section we introduce the main tools and notations we will use throughout the paper. Section \ref{s3} is devoted to state the first existence result in the case of measures that are not singular with respect to the Lebesgue measure, to introduce the approximating scheme we will use and to obtain some basic estimates on the approximating problems. In Section \ref{s4} we prove the existence result in this case, while Section \ref{sing} will be devoted to the case of a measure datum that is singular with respect to the Lebesgue measure. Finally, in the last section we provide a uniqueness result in the case $\gamma\geq 1$.

\section{Some basic facts on capacity and measures}
\subsection{Diffuse measures}

First of all, let us recall the concept of \emph{$p$-capacity} which will be useful in order to characterize the datum of our problem (for further details see \cite{hkm}).

Let $K\subseteq \Omega$ be a compact set, and
\[
W(K,\Omega)=\{ \varphi\in C^{\infty}_{0}(\Omega):\
\varphi\geq \chi_{K} \},
\]
where $\chi_{K}$ represents the characteristic function of $K$. For $p>1$, 
we define the $p$-capacity of the compact set $K$ with
respect to $\Omega$ as the quantity
\[
\capp(K)=\inf\left\{\into|\nabla
\varphi|^{p}\ dx :\varphi\in W(K,\Omega)\right\}\,.
\]

The above definition is extended in a standard way by regularity to all Borel subsets of $\Omega$.

Let us denote with $\mathcal{M}^{p}_{0}(\Omega)$ the set of all measures with
bounded variation over $\Omega$ that do not charge sets of
zero $p$-capacity, that is: if $\mu\in \mathcal{M}^{p}_{0}(\Omega)$, then
$\mu(E)=0$, for all $E\in \Omega$ such that $\capp(E)=0$.
As it is nowadays standard in the literature we will simply call \emph{diffuse} a measure $\mu$ in $\mathcal{M}^{2}_{0}(\Omega)$ while, in general, the measures in $\mathcal{M}^{p}_{0}(\Omega)$ will be called \emph{absolutely continuous} with respect to the $p$-capacity (or {\it diffuse} with respect to the $p$-capacity).

 Also notice that, as an immediate consequence of its definition, $p$-capacity enjoys the following monotonicity property with respect to the index $p$. If $1<p_{1}\leq p_{2}$, then ${\rm cap}_{p_{1}}\leq {\rm cap}_{p_{2}}$. This means, roughly speaking, that, as $p$ grows, the sets of zero $p$-capacity decrease while the number of diffuse measures increases, the threshold being $p=N$. In fact, if $p>N$ every nonempty set has positive $p$-capacity, so that $\mathcal{M}^{p}_{0}(\Omega)=\mathcal{M}(\Omega)$, the set of all bounded Radon measures on $\Omega$.

\medskip

In \cite{bgo} is proved that, if $\mu\in
\mathcal{M}^{p}_{0}(\Omega)$, then it may be decomposed as 
\begin{equation}\label{dec}
\mu=f-\div(G),
\end{equation}
where $f\in L^{1}(\Omega)$ and $G\in (L^{p'}(\Omega))^{N}$. The reverse statement is also true, that is a measure that can be decomposed this way is in $\mathcal{M}^{p}_{0}(\Omega)$. We stress the obvious fact that such a decomposition is not unique. 

Also observe that if one considers the (unique) decomposition of a measure $\mu$ in $\mathcal{M}^{p}_{0}(\Omega)$ as 
$$
\mu=\mu_{a}+\mu_{s}
$$
with $\mu_{a}$ absolutely continuous, and $\mu_{s}$ singular with respect to the Lebesgue measure then, as $\mu_{a}\in L^{1}(\Omega)$, we also have $\mu_{s} \in
\mathcal{M}^{p}_{0}(\Omega)$; moreover, if $\mu$ is nonnegative then both $\mu_{a}$ and $\mu_{s}$ are nonnegative too. Notice that the same does not happen if one looks at the decomposition \rife{dec}; in fact, even if $\mu$ is nonnegative we can not expect the term $-\div(G)$ to be nonnegative as showed in \cite[Remark 2.3]{bgo} unless $\mu_{s}\in W^{-1,p'}(\Omega)$. This remark will be useful later when we will construct suitable subsolutions to our problem. 
\subsection{Notation} Throughout the paper, if not explicitly stressed, $C$ will denote any positive real number depending only on $\Omega$, $N$, and on the data of the problem, whose value may change from line to line. For $s$ in $\mathbb{R}$, and $k > 0$, we will use the standard truncation at levels $\pm k$ defined by $T_{k}(s) = \max(-k,\min(s,k))$.

We will also use a well known consequence of Egorov theorem that we state for the convenience of the reader. 
\begin{theorem}\label{ego} Let $\{f_n \}\subset L^{1}(\Omega)$ and $\{g_n \}\subset L^{\infty}(\Omega)$ be two sequences such that $f_n$ weakly converges to $f$ in $L^{1}(\Omega)$, and 
$g_n$ converges to $g$ a.e. in and $\ast$-weakly in $L^{\infty}(\Omega)$. Then
$$
\dys\lim_{n\to \infty}\int_{\Omega} f_n\, g_n \ dx=\int_{\Omega} f g\ dx\,.
$$
\end{theorem}

\section{Main assumptions and first existence result}\label{s3}
Let $\Omega$ be a bounded open set of $\rn$, $N\geq2$. Moreover, 
let $A :\rn \mapsto \mathcal{M}^{N\times N}$ be a symmetric matrix satisfying the following
 standard assumptions: there exist two positive constants $0<\alpha\leq\beta$ such that 
\begin{equation}\label{a1}
 \alpha|\xi|^2\leq A (x)\xi\cdot\xi, \qquad \text{a.e. on} \ \Omega,\ \forall \xi\in\rn\,, 
\end{equation}
and 
\begin{equation}\label{a2}
|A (x)|\leq \beta, \qquad \text{a.e. on} \ \Omega\,.
\end{equation}
For $\gamma > 0$, and $\mu$ in $\mathcal{M}(\Omega)$, we will consider the following singular elliptic problem
\begin{equation}\label{1a}
\begin{cases}
 \dys -\div (A(x)\nabla u) = \frac{\mu}{u^{\gamma}} & \text{in}\ \Omega,\\
 u=0 &\text{on}\ \partial\Omega,\\
 u>0 &\text{on}\ \Omega\,.
 \end{cases}
\end{equation}

In order to provide an almost elementary and self-contained proof of Theorem \ref{exi} below, first we will assume that $\mu_{a}\not\equiv 0$ (in Section \ref{sing} we handle the case of purely singular measures). 
First of all we need to introduce the definition of solution for problem \rife{1a}. The main difficulty relies in giving sense to the right hand side as $u$ need not be, in general, a continuous function. 

\begin{definition}\label{defi}
Let $p>1$ and $\mu\in \mathcal{M}^{p}_{0}(\Omega)$. A function $u$ is said to be a (distributional) solution to problem \rife{1a} if $u\in W^{1,p}_{{\rm loc}}(\Omega)$, for any compact subset $\omega\subset\subset\Omega$, there exists a positive constant $c_{\omega}$ such that $$u\geq c_{\omega}>0\ \ \text{a.e. on}\ \ \omega\,,$$ 
\begin{equation}\label{def}
\into A(x)\nabla u \nabla \varphi\ dx=\into \frac{\varphi}{u^{\gamma}}d\mu, \ \ \ \text{for any}\ \ \varphi\in\mathcal{D}(\Omega)\,,
\end{equation}
and
\begin{itemize}
\item if $\gamma\leq 1$ then $u\in {W^{1,1}_{0}(\Omega)}$,
\item if $\gamma> 1$ then $u^{\frac{\gamma+1}{2}}\in \huz$.
\end{itemize}

\end{definition}

\begin{remark}
Notice that the right hand side of \rife{def} is well defined only as duality between $W^{1,p}_{0}(\Omega)\cap L^{\infty}(\Omega)$ and $L^{1}(\Omega)+W^{-1,p'}(\Omega)$. In fact, on one hand, since $\mu$ is a diffuse measure with respect to the $p$-capacity then from \rife{dec} we know that $\mu\in L^{1}(\Omega)+W^{-1,p'}(\Omega)$. On the other hand, as we require $u$ to be strictly positive on every compact subset of $\Omega$, then it is easy to check that $u^{-\gamma}\in W^{1,p}_{{\rm loc}}(\Omega)\cap L^{\infty}_{{\rm loc}}(\Omega)$, and so $\varphi u^{-\gamma}\in W^{1,p}_{0}(\Omega)\cap L^{\infty}_{}(\Omega)$. By keeping this fact in mind, with a little abuse of notation the right hand side of \rife{def} is nothing but
\begin{equation}\label{abuse}
\into \frac{\varphi}{u^{\gamma}}d\mu:=\langle\mu,\varphi u^{-\gamma}\rangle_{L^{1}(\Omega)+W^{-1,p'}(\Omega), W^{1,p}_{0}(\Omega)\cap L^{\infty}(\Omega)}\,,
\end{equation}
for every $ \varphi\in\mathcal{D}(\Omega)$.

A second important remark is the following: as we saw the use of test functions in $\mathcal{D}(\Omega)$ is useful in order to give sense to the right hand side of the equation, but it is not only a technical device. In fact, even in the regular case (e.g. $\Omega$ smooth and $\mu$ a positive H\"older continuous function on $\overline{\Omega}$) the right hand side of the equation is not integrable in general (unless $\gamma<1$) (see \cite{lm}). 
\end{remark}

\begin{remark}\label{augu}
Some words on the boundary data are in order. If $\gamma>1$ then the boundary datum is achieved in a weaker sense than the usual one. Anyway let us observe that, if $\Omega$ is smooth enough and $u^{\frac{\gamma+1}{2}}\in \huz$, then
$$
\lim_{\vare\to 0^{+}}\frac{1}{\vare}\int_{\{x:{\rm dist}(x,\partial\Omega)<\vare\}} u^{\frac{\gamma+1}{2}} (x) \, dx=0
$$ 
(see for instance \cite{ponce}). Now, as $u$ is nonnegative and $\gamma>1$, using H\"older's inequality one can easily get that 
\begin{equation}\label{aug}
\lim_{\vare\to 0^{+}}\frac{1}{\vare}\int_{\{x:{\rm dist}(x,\partial\Omega)<\vare\}} u (x) \, dx=0,
\end{equation} 
which is a clearer way to understand the boundary condition and that will be used in Section \ref{u} in order to prove uniqueness. 
\end{remark}

Now we are in the position to state our main existence result.

\begin{theorem}\label{exi}
Let either $\gamma\geq1$ and $\mu$ be a nonnegative diffuse measure with respect to the $2$-capacity or $0<\gamma<1$ and $\mu$ be a nonnegative diffuse measure with respect to the $q'$-capacity with $q'=\frac{N(\gamma+1)}{(N-1)\gamma+1}$. Moreover assume $\mu_{a}\not\equiv 0$. Then there exists a solution for problem \rife{1a} in the sense of Definition \ref{defi}. 
\end{theorem}
\begin{remark}\label{optimal}
Notice that the result of Theorem \ref{exi} perfectly fits with the nonexistence result in \cite{bo}. In fact, in that paper the authors prove nonexistence of solutions for problem \rife{1a} (in the sense of approximating sequences) if either $\gamma\geq1$ and $\mu$ is concentrated on a set of zero $2$-capacity or $0<\gamma<1$ and $\mu$ is concentrated on a set of zero $q'$-capacity with $q'=\frac{N(\gamma+1)}{(N-1)\gamma+1}$. Also observe that, if $\gamma<1$ then $q'>2$ so that, due to the properties of the capacities, we are allow to consider a larger set (with respect to the case $\gamma\geq 1 $) of measure data. 
\end{remark}
\subsection{Approximating solutions and basic estimates} 

We need to approximate the measure $\mu$ with a sequence of smooth functions. The following approximation result can be found for instance in \cite{bgo} (see also \cite{mp}).
\begin{proposition}\label{mp}
Let $\mu=f-\div(G)$ be a nonnegative measure in $\mathcal{M}^{p}_{0}(\Omega)$. Then there exists a sequence of nonnegative functions $\mu_{n}\in L^{2}(\Omega)$ such that
$$
\mu_{n}=f_{n}-{\text\rm \div}(G_{n})\ \ \text{in}\ \ \mathcal{D}'(\Omega)\,,\ \ \|\mu_{n}\|_{L^{1}(\Omega)}\leq C\,,
$$
where $f_{n}$ belongs to $L^{2}(\Omega)$ and weakly converges to $f$ in $L^{1}(\Omega)$ and $G_{n}$ strongly converges to $G$ in $(L^{p'}(\Omega))^{N}$. 
\end{proposition}

In order to construct suitable subsolutions for our approximating problems we need to approximate the datum $\mu$ in a particular way. Consider the decomposition of $\mu$ with respect to the Lebesgue measure
$$
\mu=\mu_{a}+\mu_{s}, 
$$
and let $\mu_{s,n}$ be an approximation of the singular part of $\mu$ given as in Proposition \ref{mp} (remark that $0 \leq \mu_{s}\leq \mu$ and so $\mu_{s}\in \mathcal{M}^{p}_{0}(\Omega)$).

Let $u_{n}$ be a solution to 
\begin{equation}\label{pn}
\begin{cases}
 \dys -\div(A(x)\nabla u_{n}) = \frac{T_{n}(\mu_{a})+\mu_{s,n}}{\left(\frac1n+u_{n}\right)^{\gamma}} & \text{in}\ \Omega,\\
u_{n}=0 &\text{on}\ \partial\Omega\,.
 \end{cases}
\end{equation}
Existence of positive distributional solutions for problem \rife{pn} can be deduced by standard fixed point argument as in \cite{bo}. Moreover, depending on the value of $\gamma$, the sequence $\{u_{n}\}$ satisfies some a priori estimates that we collect in the following

\begin{proposition}\label{stime}
Let $u_{n}$ be a solution to problem \rife{pn}. Then 
\begin{enumerate}
\item If $\gamma<1$, then 
$$
\|u_{n}\|_{W^{1,q}_{0}(\Omega)}\leq C, \ \ \text{where}\ q=\frac{N(\gamma+1)}{N-1+\gamma}\,.
$$
\item If $\gamma=1$, then 
$$
\|u_{n}\|_{\huz}\leq C\,.
$$
\item If $\gamma>1$, then 
$$
\|u_{n}\|_{H^{1}(\omega)}\leq C_{\omega}\ \ \forall\ \omega\subset\subset\Omega, \ \text{and}\ \ \|u_{n}^{\frac{\gamma+1}{2}}\|_{\huz}\leq C\,.
$$
\end{enumerate}
\end{proposition}
\begin{proof}
The proof of this result can be easily deduced from, respectively, Theorem 5.6, Lemma 3.1, and Lemma 4.1 in \cite{bo}.
\end{proof}

Now we state and prove the key result of this section, that will allow us to pass to the limit in the approximating problems. 

\begin{lemma}\label{pos}
Let $u_{n}$ be the solutions of problem \rife{pn} with $\mu_{a}\not= 0$. Then, for any $\omega\subset\subset\Omega$ there exists a positive constant $c_{\omega}$ such that
\begin{equation}\label{bb}
u_{n}\geq c_{\omega}>0, \ \forall\ n\in \na\,.
\end{equation}
\end{lemma}
\begin{proof}
Let $w$ be the (unique) weak solution of problem
$$
\begin{cases}
 \dys -\div(A(x) \nabla w) = \frac{T_{1}(\mu_{a})}{(1 + w)^{\gamma}} & \text{in}\ \Omega,\\
 w=0 &\text{on}\ \partial\Omega\,.
 \end{cases}
$$
In \cite{bo} (Lemma 2.1) it is proved that $w\in \huz\cap L^{\infty}(\Omega)$; moreover $w\geq c_{\omega}>0$ on every compactly supported $\omega\subset\Omega$.

Then, for any $n$, we have that 
\begin{equation}\label{subbe}
 -\div(A(x) \nabla w) \leq \frac{T_{n}(\mu_{a})}{(\frac1n +w)^{\gamma}}\leq \frac{T_{n}(\mu_{a})+\mu_{s,n}}{\left(\frac1n +w\right)^{\gamma}}\,,
\end{equation}
where we have used that $\mu_{s,n}\geq 0$. 
Now we take $(w-u_{n})^{+}$ as test function in both \rife{pn} and \rife{subbe}, we subtract the two formulations and we use ellipticity in order to get
$$
\begin{array}{l}
\dys \alpha\into |\nabla (w-u_{n})^{+}|^{2}\ dx \\ \\ \dys \leq \into (T_{n}(\mu_{a})+\mu_{s,n})\left(\frac{1}{\left(\frac1n +w\right)^{\gamma}}- \frac{1}{\left(\frac1n + u_{n}\right)^{\gamma}}\right)(w-u_{n})^{+}\leq 0\,,
\end{array}
$$
from which we deduce $u_{n}\geq w$ a.e. on $\Omega$, that is, for any $\omega\subset\subset\Omega$, there exists $c_{\omega}>0$, such that 
$$
u_{n}\geq w\geq c_{\omega}>0\,,
$$
and this concludes the proof of $\rife{bb}$. 
\end{proof}

\section{Proof of Theorem \ref{exi}}\label{s4}

\begin{proof}[Proof of Theorem \ref{exi}]

In order to deal with the different regularity and convergences that one derives from Proposition \ref{stime} we need to distinguish between various values of $\gamma$. 

{\it Proof of Theorem \ref{exi} in the case $\gamma= 1$.} Recall that, by Proposition \ref{mp},
$$
\mu_{s,n}=f_{n}-{\text\rm \div}(G_{n})\ \ \text{in}\ \ \mathcal{D}'(\Omega)\,,$$
where $f_{n}$ belongs to $L^{2}(\Omega)$ and weakly converges to $f$ in $L^{1}(\Omega)$ and $G_{n}$ strongly converges to $G$ in $(L^{2}(\Omega))^{N}$, with $\mu_{s}=f-\div(G)$. 

Moreover, using Proposition \ref{stime} and Lemma \ref{pos} we easily deduce the existence of a function $u\in \huz$ such that, up to subsequences, 
$$
 u_{n}\longrightarrow u \ \ \text{a.e. and weakly in}\ \huz\,,
 $$
 and
$$\left (\frac1n+u_{n}\right)^{-1} \longrightarrow \frac1u \ \ \text{a.e. and $\ast$-weakly in $L^{\infty}(\omega)$,}\ \forall \omega\subset\subset\Omega\,.
$$

We need to pass to the limit in the distributional formulation of \rife{pn}; so, if $\varphi\in \mathcal{D}(\Omega)$, we have
\begin{equation}\label{pass}
\begin{array}{l}
\dys\into A(x)\nabla u_{n} \nabla \varphi\ dx=\into \frac{T_{n}(\mu_{a})\varphi}{\frac1n +u_{n}}\ dx \\ \qquad \dys+\into \frac{f_{n}\varphi}{\frac1n +u_{n}}\ dx+\left\langle G_{n}, \nabla \left(\frac{\varphi}{\frac1n +u_{n}}\right)\right\rangle.
\end{array}
\end{equation}

Thanks to the convergence results we proved, there are no difficulties in passing to the limit in all terms of \rife{pass} but the last one. Only notice that in the convergence of the second integral on the right hand side of \rife{pass} we also need to use Theorem \ref{ego} as we only ask for a weak convergence in $L^{1}(\Omega)$ for the sequence $\{f_{n}\}$. 
In order to pass to the limit in the last term, observe that
$$
 \nabla \left(\frac{\varphi}{\frac1n +u_{n}}\right)=\frac{\nabla\varphi}{\frac1n +u_{n}}- \frac{\nabla u_{n}}{\left(\frac1n+u_{n}\right)^{2}}\varphi, 
$$
and, while the first term on the right hand side is strongly convergent in $(L^{2}(\Omega))^{N}$, the second one is weakly convergent in $(L^{2}(\Omega))^{N}$ as $\varphi$ has compact support in $\Omega$ and by Theorem \ref{ego}. This is enough, recalling \rife{abuse}, in order to pass to the limit in \rife{pass} and to get 
$$
\into A(x)\nabla u \nabla \varphi\ dx=\into \frac{\varphi}{u}d\mu, \ \ \ \text{for any}\ \ \varphi\in\mathcal{D}(\Omega). 
$$

{\it Proof of Theorem \ref{exi} in the case $\gamma>1$.} The proof in this case is very similar, up to localization, to the previous one.

Again we have 
$$
\mu_{s,n}=f_{n}-{\text\rm \div}(G_{n})\ \ \text{in}\ \ \mathcal{D}'(\Omega)\,,$$
where $f_{n}$ belongs to $L^{2}(\Omega)$ and weakly converges to $f$ in $L^{1}(\Omega)$ and $G_{n}$ strongly converges to $G$ in $(L^{2}(\Omega))^{N}$, with $\mu_{s}=f-\div(G)$. 

Now, using Proposition \ref{stime} and Lemma \ref{pos} we deduce the existence of a function $u\in H^{1}_{{\rm loc}}(\Omega)$ such that $u^{\frac{\gamma+1}{2}}\in \huz$, and such that, up to subsequences, 
$$
 u_{n}\longrightarrow u \ \ \text{a.e. and weakly in}\ H^{1}(\omega)\,, \ \forall \omega\subset\subset\Omega\,,
 $$
 and
$$\left (\frac1n+u_{n}\right)^{-\gamma} \longrightarrow \frac{1}{u^{\gamma}} \ \ \text{a.e. and $\ast$-weakly in $L^{\infty}(\omega)$,}\ \forall \omega\subset\subset\Omega\,.
$$

If $\varphi\in \mathcal{D}(\Omega)$, there are no problems in passing to the limit in all terms of the distributional formulation of \rife{pn}, that is in
$$
\begin{array}{l}
\dys\into A(x)\nabla u_{n} \nabla \varphi\ dx=\into \frac{T_{n}(\mu_{a})\varphi}{(\frac1n +u_{n})^{\gamma}}\ dx \\ \qquad \dys+\into \frac{f_{n}\varphi}{(\frac1n +u_{n})^{\gamma}}\ dx+\left\langle G_{n}, \nabla \left(\frac{\varphi}{(\frac1n +u_{n})^{\gamma}}\right)\right\rangle.
\end{array}
$$

Only observe that in this case
$$
 \nabla \left(\frac{\varphi}{(\frac1n +u_{n})^{\gamma}}\right)=\frac{\nabla\varphi}{(\frac1n +u_{n})^{\gamma}}- \gamma\frac{\nabla u_{n}}{\left(\frac1n+u_{n}\right)^{\gamma+1}}\varphi, 
$$
and we can apply again Theorem \ref{ego} in order to deal with the second term. We then pass to the limit in \rife{pass} and we get
$$
\into A(x)\nabla u \nabla \varphi\ dx=\into \frac{\varphi}{u^{\gamma}}d\mu, \ \ \ \text{for any}\ \ \varphi\in\mathcal{D}(\Omega). 
$$

{\it Proof of Theorem \ref{exi} in the case $\gamma<1$.} Now, let $q=\frac{N(1+\gamma)}{N-1+\gamma}$. In this case, using Proposition \ref{mp} we have that
$$
\mu_{s,n}=f_{n}-{\text\rm \div}(G_{n})\ \ \text{in}\ \ \mathcal{D}'(\Omega)\,,$$
where $f_{n}$ belongs to $L^{2}(\Omega)$ and weakly converges to $f$ in $L^{1}(\Omega)$ and $G_{n}$ strongly converges to $G$ in $(L^{q'}(\Omega))^{N}$, with $\mu_{s}=f-\div(G)$. 

From Proposition \ref{stime} and Lemma \ref{pos} we now deduce that there exists a function $u\in W^{1,q}_{0}(\Omega)$ such that, up to subsequences, 
$$
 u_{n}\longrightarrow u \ \ \text{a.e. and weakly in}\ W^{1,q}_{0}(\Omega)\,, \ \forall \omega\subset\subset\Omega\,,
 $$
 and
$$\left (\frac1n+u_{n}\right)^{-\gamma} \longrightarrow \frac{1}{u^{\gamma}} \ \ \text{a.e. and $\ast$-weakly in $L^{\infty}(\omega)$,}\ \forall \omega\subset\subset\Omega\,.
$$

If $\varphi\in \mathcal{D}(\Omega)$, we pass to the limit in the distributional formulation of \rife{pn}, as before 
and we get
$$
\into A(x)\nabla u \nabla \varphi\ dx=\into \frac{\varphi}{u^{\gamma}}d\mu, \ \ \ \text{for any}\ \ \varphi\in\mathcal{D}(\Omega). 
$$
In fact, observe that we have
$$
 \nabla \left(\frac{\varphi}{(\frac1n +u_{n})^{\gamma}}\right)=\frac{\nabla\varphi}{(\frac1n +u_{n})^{\gamma}}- \gamma\frac{\nabla u_{n}}{\left(\frac1n+u_{n}\right)^{\gamma+1}}\varphi, 
$$
and we can use again Theorem \ref{ego} in order to get the weak convergence of the second term in $(L^{q}(\Omega))^{N}$.

\end{proof}

\section{The case $\mu\perp\mathcal{L}$} 
\label{sing}

This section is devoted to the study of the case $\mu\perp\mathcal{L}$; that is, $\mu_{a}=0$. This situation is much more delicate and can not be faced with the elementary techniques of the previous sections. We will develop a refined monotonicity argument which is based on a monotone approximation for diffuse measures. We also stress that the following argument, with suitable modifications, works for any diffuse measure (in the sense specified in Theorem \ref{exi}) not only the ones concentrated on a set of zero Lebesgue measure. We will prove existence of a solution in full generality in the case $\gamma\geq 1$, while in the case $\gamma<1$ there is a ``gap'' between what one expects, and what we can prove. 

\medskip 

So, given $\mu\perp\mathcal{L}$, we consider problem
\begin{equation}\label{gammaTOT}
\begin{cases}
 \dys -\div (A(x)\nabla u) = \frac{\mu}{u^{\gamma}} & \text{in}\ \Omega,\\
 u=0 &\text{on}\ \partial\Omega,
 \end{cases}
\end{equation}
where $\Omega$ is an open bounded set of $\rn$, $A$ is an elliptic matrix satisfying \rife{a1} and \rife{a2}, and $\gamma>0$.

We will use several technical results. The first one is a corollary of a result in \cite{dm} (see also \cite{bgo}) that we state as in \cite[Lemme 4.2]{bp} (see also \cite{ponce}). 
\begin{lemma}\label{bapi}
Let $\mu$ be a nonnegative diffuse measure with respect to the $p$-capacity. Then there exists an increasing sequence of nonnegative increasing measures $\mu_{n}\in W^{-1,p'}(\Omega)$ such that $\mu_{n}$ converges to $\mu$ strongly in $\mathcal{M}(\Omega)$. 
\end{lemma}

We will also use the following consequence of a standard capacitary result that can be found in \cite[Proposition 2.7]{dmop} (see also \cite{hkm}). Recall that, if $u\in W^{1,p}(\Omega)$ then one can consider its ${\rm cap}_{p}$-quasicontinuous representative $\tilde{u}$: a function which satisfies 
\begin{itemize}
\item[$a)$] $u=\tilde{u}$ a.e. on $\Omega$,
\item[$b)$] for all $\vare>0$ there exists a set $E$ such that ${\rm cap}_{p}(E)<\vare$ and $\tilde{u}$ is continuous on $\Omega\backslash E$. 
\end{itemize}

From now on we will always refer to the ${\rm cap}_{p}$-quasicontinuous representative of the involved Sobolev functions. 
 
\begin{lemma}\label{bdd}
Let $\mu$ be a nonnegative diffuse measure with respect to the $p$-capacity, and let $u\in W^{1,p}_{0}(\Omega)\cap L^{\infty} (\Omega)$ be a nonnegative function. Then, up to the choice of its cap$_{p}$-quasicontinuous representative, $u$ belongs to $L^{\infty}(\Omega, \mu)$ and
$$
\int_{\Omega}u\ d\mu\leq \|u\|_{L^{\infty}(\Omega)}\mu(\Omega)\,.
$$
\end{lemma}

Here is the main existence result of this section:
\begin{theorem}\label{exiTOT}
 Let either $\gamma\geq 1$ and $\mu$ be a diffuse measure with respect to the $2$-capacity or $0<\gamma<1$ and $\mu$ be a diffuse measure with respect to the $q$-capacity, with $q=\frac{N(\gamma+1)}{N-1+\gamma}$ . Then there exists a distributional solution for problem \rife{gammaTOT}.
\end{theorem}
 
\begin{remark}
The result in the case $\gamma < 1$ is not optimal. In fact, in the case of a measure $\mu$ concentrated on a set of zero Lebesgue measure, we need to restrict the set of admissible measures to those in $\mathcal{M}^{q}_{0}(\Omega)$ rather than $\mathcal{M}^{q'}_{0}(\Omega)$ in order to prove our existence result. This reflects a standard difficulty which arises from a lack of suitable regularity results when the datum involves measures in $W^{-1,q}(\Omega)$ with $q < 2$. For some further comments on this fact we refer the reader to \cite{bo97} (and also \cite{orpr}).
\end{remark}

\begin{proof}[Proof of Theorem \ref{exiTOT}]
As in the proof of Theorem \ref{exi} we will distinguish between the three cases $\gamma = 1$, $\gamma > 1$ and $\gamma < 1$, adopting the strategy of approximating the problem, proving a strict positivity result for the solutions of these problems, as well as some {\sl a priori} estimates, and then passing to the limit.

In what follows, we define
$$
q(\gamma) = 
\begin{cases}
\hfill 2 \hfill & \mbox{if $\gamma \geq 1$,} \\
\hfill \frac{N(\gamma+1)}{N-1+\gamma} \hfill & \mbox{if $0 < \gamma < 1$,}
\end{cases}
\qquad
q'(\gamma) = 
\begin{cases}
\hfill 2 \hfill & \mbox{if $\gamma \geq 1$,} \\
\hfill \frac{N(\gamma+1)}{(N-1)\gamma + 1} \hfill & \mbox{if $0 < \gamma < 1$.}
\end{cases}
$$

Let $\mu$ be a nonnegative diffuse measure with respect to the $q(\gamma)$ capacity, and let $\{\mu_{n}\}$ be an increasing sequence of nonnegative measures in $W^{-1,q'(\gamma)}$, given by Lemma \ref{bapi}, strongly converging to $\mu$. Since $q'(\gamma) \geq 2$, thanks to Schauder theorem it can be easily proved that for every $n$ in $\mathbb{N}$ there exists a (unique) positive solution $u_{n}$ in $\huz$ of
\begin{equation}\label{gammagamma}
\begin{cases}
\dys
-\div (A(x)\nabla u_{n}) = \frac{\mu_{n}}{\big(\frac{1}{n} + u_{n})^{\gamma}} & \text{in}\ \Omega,\\
u_{n}=0 & \text{on}\ \partial\Omega\,.
\end{cases}
\end{equation}

We begin by choosing $v = (u_{n} - u_{n+1})^{+}$ as test function in \eqref{gammagamma}, a choice which is possible since $v$ belongs to $\huz$. We obtain, after subtracting the equations for $u_{n}$ and $u_{n+1}$, using \eqref{a1}, and recalling that $\mu_{n}$ is increasing,
$$
\begin{array}{l}
\dys
\alpha \into |\nabla (u_{n}-u_{n+1})^{+}|^{2} dx \\
\dys
\quad
\leq
\into \left( \frac{d\mu_{n}}{(\frac1n + u_{n})^{\gamma}} - \frac{d\mu_{n+1}}{(\frac{1}{n+1} + u_{n+1})^{\gamma}}\right) (u_{n}-u_{n+1})^{+} \\
\dys
\quad
\leq
\into \left( \frac{1}{(\frac1n + u_{n})^{\gamma}} - \frac{1}{(\frac{1}{n+1} + u_{n+1})^{\gamma}}\right)(u_{n}-u_{n+1})^{+}\ d\mu_{n+1} \\
\dys
\quad
\leq \into \left( \frac{1}{(\frac{1}{n+1} + u_{n})^{\gamma}} - \frac{1}{(\frac{1}{n+1} + u_{n+1})^{\gamma}}\right)(u_{n}-u_{n+1})^{+}\ d\mu_{n+1}
\leq 0\,.
\end{array}
$$
Thus, we have proved that
\begin{equation}\label{increase}
\mbox{$0 \leq u_{n} \leq u_{n+1}$ for every $n$ in $\mathbb{N}$.}
\end{equation}

We now remark that $\mu_{1} / (1 + u_{1})^{\gamma}$ is not identically zero and nonnegative. Then, by the strong maximum principle we have that, for every $\omega\subset\subset\Omega$, there exists $c_{\omega} > 0$ such that $u_{1}\geq c_{\omega}$ in $\omega$. Therefore, we have proved that
\begin{equation}\label{dasotto}
\mbox{for every $\omega\subset\subset\Omega$ there exists $c_{\omega} > 0:\ u_{n}\geq c_{\omega}$ in $\omega$, for every $n$ in $\mathbb{N}$.}
\end{equation}

We now turn to {\sl a priori} estimates.

If $\gamma = 1$, choosing $u_{n}$ as test function (which can be done since $u_{n}$ belongs to $\huz$, and the right hand side to the dual space $H^{-1}(\Omega)$) we obtain, after using \eqref{a1}, that
$$
\alpha \int_{\Omega} |\nabla u_{n}|^{2} 
\leq
\int_{\Omega} \frac{u_{n}}{\frac1n + u_{n}}d\mu_{n}
\leq
\mu_{n}(\Omega)
\leq C\,,
$$
where in the last passage we have used Lemma \ref{bdd}.
Therefore, we have proved that
\begin{equation}\label{aprioug1}
\mbox{the sequence $\{u_{n}\}$ is bounded in $\huz = W^{1,q(1)}_{0}(\Omega)$.}
\end{equation}

If $\gamma > 1$, let $k > 0$ and choose $v = T_{k}(u_{n})^{\gamma}$ as test function in \eqref{gammagamma}; note that this choice is admissible since $v$ belongs to $\huz$ and the right hand side is in the dual space $H^{-1}(\Omega)$. Using again \eqref{a1} and Lemma \ref{bdd}, we have
$$
\alpha\gamma\into|\nabla T_{k}(u_{n})|^{2}T_{k}(u_{n})^{\gamma-1}\,dx \leq \into \frac{T_{k}(u_{n})^{\gamma}}{(\frac1n+u_{n})^{\gamma}}\ d\mu_{n}
\leq
\mu_{n}(\Omega) \leq C\,. 
$$
Thus, we have proved that
$$
\frac{4\alpha\gamma}{(\gamma+1)^{2}}\into |\nabla T_{k}(u_{n})^{\frac{\gamma+1}{2}}|^{2} dx\leq C\,,
$$
which implies, letting $k$ tend to infinity, that
$$
\into |\nabla u_{n}^{\frac{\gamma+1}{2}}|^{2} dx\leq C\,.
$$
The proof of the $H^{1}_{{\rm loc}}$ estimate now follows exactly as in \cite[Lemma 4.1]{bo} with straightforward modifications. Only observe that, in order to handle the source term we use Lemma \ref{bdd} and \eqref{dasotto} to have that 
$$
\into \frac{u_{n}\varphi^{2}}{(\frac1n+u_{n})^{\gamma}}\ d\mu_{n}\leq \frac{1}{c_{\omega}^{\gamma-1}}\mu(\Omega)\,, 
$$
where ${\rm supp}(\varphi)\subset\omega\subset\subset\Omega$. Therefore, we have proved that
\begin{equation}\label{aprioma1}
\mbox{the sequence $\{u_{n}\}$ is bounded in $H^{1}_{{\rm loc}}(\Omega) = W^{1,q(\gamma)}_{{\rm loc}}(\Omega)$.}
\end{equation}

If $\gamma < 1$, we let $0 < \vare < \frac1n$ and choose $v = (\vare+u_{n})^{\gamma}-\vare^{\gamma}$ as test function in \eqref{gammagamma}; this is allowed as $v$ belongs to $\huz$, and $\mu_{n}$, being in $W^{-1,q'(\gamma)}(\Omega)$, with $q'(\gamma) > 2$, is in $H^{-1}(\Omega)$. Using \eqref{a1} as before, as well as Lemma \ref{bdd}, we obtain
\begin{equation}\label{5.2}
\alpha\gamma\into |\nabla u_{n}|^{2} (\vare+u_{n})^{\gamma-1}\ dx
\leq
\into \frac{(\vare+u_{n})^{\gamma}-\vare^{\gamma}}{(\frac1n +u_{n})^{\gamma}}\ d\mu_{n}
\leq
\mu_{n}(\Omega) \leq C\,.
\end{equation}
Now we apply Sobolev inequality and then let $\vare$ tend to zero to obtain
$$
\into u_{n}^{\frac{2^{\ast}(\gamma+1)}{2}}\leq C\,, 
$$
so that $\{u_{n}\}$ is bounded in $L^{s}(\Omega)$, with $s = \frac{N(\gamma+1)}{N-2}$.

Coming back to \rife{5.2} we use H\"{o}lder inequality to obtain
$$
\begin{array}{l}
\dys
\into |\nabla u_{n}|^{q(\gamma)}\ dx
=
\into \frac{ |\nabla u_{n}|^{q(\gamma)}}{(\vare+u_{n})^{\frac{q(\gamma)(1-\gamma)}{2}}}(\vare+u_{n})^{\frac{q(\gamma)(1-\gamma)}{2}} \\
\dys
\quad
\leq
\left( \into |\nabla u_{n}|^{2}(\vare +u_{n})^{\gamma-1} \right)^{\frac{q(\gamma)}{2}}\left(\into (\vare+u_{n})^{s}\right)^{1-\frac{q(\gamma)}{2}}\leq C\,.
\end{array}
$$
Thus, we have proved that
\begin{equation}\label{apriomi1}
\mbox{the sequence $\{u_{n}\}$ is bounded in $W^{1,q(\gamma)}_{0}(\Omega)$.}
\end{equation}

Summing up the results of \eqref{aprioug1}, \eqref{aprioma1}, and \eqref{apriomi1}, we have that $\{u_{n}\}$ is bounded in $W^{1,q(\gamma)}_{\rm loc}(\Omega)$ (and actually more, if $\gamma \leq 1$). Thus, also using \eqref{increase}, $u_{n}$ weakly converges to some positive function $u$ in $W^{1,q(\gamma)}_{\rm loc}(\Omega)$. Choosing $\vp$ in $\mathcal{D}(\Omega)$ as test function in \eqref{gammagamma}, we have that
$$
\into A(x) \nabla u_{n} \cdot \nabla \vp
=
\into \frac{\vp}{(\frac1n + u_{n})^{\gamma}}d\mu_{n}\,.
$$
Using the (local) weak convergence of $\nabla u_{n}$ in $L^{q(\gamma)}(\Omega)$, we have that
$$
\lim_{n \to +\infty} \into A(x) \nabla u_{n} \cdot \nabla \vp
=
\into A(x) \nabla u \cdot \nabla \vp\,,
$$
so that it only remains to pass to the limit in the right hand side to conclude the proof of the theorem.

We have
$$
\into\frac{\varphi }{(\frac1n + u_{n})^{\gamma}}\ d\mu_{n} = \into\frac{\varphi}{(\frac1n + u_{n})^{\gamma}}\ d(\mu_{n}-\mu)+ \into\frac{\varphi }{(\frac1n + u_{n})^{\gamma}}\ d\mu\,,
$$
and we will separately analyze the two terms on the right hand side of the above identity. First of all, if $\omega$ is an open set compactly contained in $\Omega$ and containing the support of $\vp$, thanks to both Lemma \ref{bapi} and Lemma \ref{bdd}, and to \eqref{dasotto}, we have
$$
\into\frac{\varphi}{(\frac1n + u_{n})^{\gamma}}\ d( \mu_{n}-\mu)\leq \frac{1}{c_{\omega}^{\gamma}}\|\varphi\|_{L^{\infty}(\Omega)}|\mu_{n} - \mu|(\Omega)\stackrel{n\to\infty}{\longrightarrow} 0\,.
$$
On the other hand, since $\mu$ is diffuse with respect to the $q(\gamma)$ capacity, we know that $\mu=f-\div(G)$, 
with $f$ in $L^{1}(\Omega)$ and $G$ in $(L^{q'(\gamma)}(\Omega))^{N}$. Therefore,  reasoning as in the proof of Theorem \ref{exi}, we deduce that $\frac{\varphi}{(\frac1n + u_{n})^{\gamma}}$ converges to $\frac{\varphi}{u^{\gamma}}$ both $\ast$-weakly in $L^{\infty}(\Omega)$ and weakly in $W^{1,q(\gamma)}_{0}(\Omega)$, so that we have 
$$
\lim_{n\to+\infty}\into\frac{\varphi}{(\frac1n + u_{n})^{\gamma}}\ d\mu=\into\frac{\varphi}{u^{\gamma}}\ d\mu\,,
$$
and this concludes the proof.
\end{proof}

\section{A general uniqueness result if $\gamma\geq 1$}\label{u}

In \cite{lm} the uniqueness of a $C^{2+\alpha}(\Omega)\cap C(\overline{\Omega})$ solution for problem 
\begin{equation}\label{unica}
\begin{cases}
 \dys -\Delta u = \frac{\mu}{u^{\gamma}} & \text{in}\ \Omega,\\
 u=0 &\text{on}\ \partial\Omega,\\
 u>0 &\text{on}\ \Omega,
 \end{cases}
\end{equation}
was proved if the case of a strictly positive $\mu\in C^{\alpha}(\overline{\Omega})$ if $\Omega$ is a domain of class $C^{2+\alpha}$, for some $0<\alpha<1$. In our case, for arbitrary $\gamma$, there is no chance to deal with such smooth solutions since, as we have seen, not even local finite energy solutions can be obtained, at least in the case $\gamma<1$. If $\gamma \geq 1$ we show a general uniqueness result for the model problem \rife{unica} in the case of $\mu$ being a diffuse measure. In order to deal also with the case $\gamma>1$ we assume some smoothness on $\Omega$ (e.g. Lipschitz boundary). The proof is based on a Kato's type inequality. 

\begin{theorem}
Let $\gamma \geq 1$ and let $\mu$ be as in Theorem \ref{exi}. Then the solution to problem \rife{unica} is unique. 
\end{theorem} 
\begin{proof}
Let $u$ and $v$ be two solutions to problem \rife{unica} and consider $w=u-v$. We will prove that $u\leq v$, the reverse inequality being similar. Consider 
 a nonnegative $\varphi\in \mathcal{D}(\Omega)$, and choose $\frac{1}{\vare}T_{\vare}(w)^{+}\varphi$ as test function in the weak formulations of both $u$ and $v$; subtracting, we obtain
$$
\begin{array}{l}
\dys\frac{1}{\vare}\int_{\{w^{+}\leq\vare\}} |\nabla w|^{2}\varphi\ dx +\frac{1}{\vare}\into \nabla w\nabla\varphi T_{\vare}(w)^{+}\ dx \\ \qquad
\dys =\frac{1}{\vare}\into \mu\varphi \left(\frac{1}{u^{\gamma}}-\frac{1}{v^{\gamma}}\right)T_{\vare}(w)^{+}\ dx\,;
\end{array}
$$
now we drop the first nonnegative term and we let $\vare\to 0$ so that, observing that $(u^{-\gamma}-v^{-\gamma}){\rm sign} (w)\leq 0$, we get
\begin{equation}\label{distri}
-\Delta w^{+}\leq 0\ \ \text{in}\ \ \mathcal{D}'(\Omega)\,.
\end{equation}
If $\gamma=1$, this proves the result as $w\in W^{1,1}_{0}(\Omega)$ (see for instance \cite[Proposition 4.B.1]{bmp}).

If $\gamma>1$, recalling \rife{aug} in Remark \ref{augu}, we deduce
$$
\lim_{\vare\to 0^{+}}\frac{1}{\vare}\int_{\{x:{\rm dist}(x,\partial\Omega)<\vare\}} w^{+} (x) \, dx=0,
$$
that together with \rife{distri} allows us to apply Proposition 5.2 of \cite{ponce} in order to get that 
$$
-\Delta w^{+}\leq 0\ \ \text{in}\ \ (C^{\infty}_{0}(\overline{\Omega}))'\,,
$$
where $C^{\infty}_{0}(\overline{\Omega})$ is the set of functions in $C^{\infty}(\overline{\Omega})$ that vanish on $\partial\Omega$.

We can now apply standard weak maximum principle (see for instance Proposition 5.1 of \cite{ponce}) and we get that $w^{+}=0$. This concludes the proof of the result. 
\end{proof}

\begin{remark}
We stress again that uniqueness results for singular problems as the one we considered are very hard to prove in general. Some results in this direction can be found in the recent paper \cite{bcd} (see also \cite{olpe}). 
A trivial remark is that the previous proof works also for $\gamma<1$ in the class of $H^{1}_{\rm loc}$ functions: that is, if $\gamma<1$, then $H^{1}_{\rm loc}$ solutions to problem \rife{unica} (if they exist) are unique. 
\end{remark}

%
%

\end{document}